\documentclass[headsepline,12pt,a4paper]{amsart}
\usepackage{amsmath,amssymb,amsthm,exscale,calc}
\usepackage[latin1]{inputenc}
\usepackage{latexsym,textcomp}
\usepackage{graphicx}
\usepackage{parskip}
\usepackage{floatflt}
\usepackage{picins}
\theoremstyle{definition}
\newtheorem{Thm}{Theorem}

\newtheorem{Lemma}{Lemma}

\newtheorem{Proposition}{Proposition}

\newtheorem*{Remark}{Remark}
\newtheorem*{Example}{Examples}

\newcommand{\Z}{{\mathbb Z}}
\newcommand{\R}{{\mathbb R}}
\newcommand{\C}{{\mathbb C}}
\newcommand{\Q}{{\mathbb Q}}
\newcommand{\N}{{\mathbb N}}

\newcommand{\D}{{\mathbb D}}
\newcommand{\Sp}{{\mathbb S}}

\newcommand{\Lp}{{\Delta}}
\newcommand{\LF}{\widetilde{\Delta}}
\newcommand{\se}{\sigma_{\!\mathrm{ess}}}

\newcommand{\inn}{{\mathrm {int}}}
\newcommand{\dd}{{\partial}}
\newcommand{\al}{{\alpha}}
\newcommand{\be}{{\beta}}
\newcommand{\de}{{\delta}}
\newcommand{\vareps}{{\varepsilon}}
\newcommand{\eps}{{\varepsilon}}
\newcommand{\ka}{{\kappa}}
\newcommand{\ph}{{\varphi}}
\newcommand{\lm}{{\lambda}}
\newcommand{\gm}{{\gamma}}
\newcommand{\si}{{\sigma}}

\newcommand{\qand}{{\quad\mathrm {and}\quad}}

\newcommand{\av}[1]{\left\vert #1\right\vert}

\newcommand{\ab}[1]{\left( #1\right)}
\newcommand{\ac}[1]{\left\{ #1\right\}}

\newcommand{\ov}[1]{\overline{#1}}

\newcommand{\Hm}[1]{\leavevmode{\marginpar{\tiny%
$\hbox to 0mm{\hspace*{-0.5mm}$\leftarrow$\hss}%
\vcenter{\vrule depth 0.1mm height 0.1mm width \the\marginparwidth}%
\hbox to 0mm{\hss$\rightarrow$\hspace*{-0.5mm}}$\\\relax\raggedright
#1}}}

\begin{document}

\title[Curvature, geometry and spectrum of
planar graphs]
{Curvature, geometry and spectral properties of
planar graphs}
\author[Matthias~Keller]{Matthias Keller}
\address[Matthias~Keller]{Mathematical institute, Friedrich-Schiller-University Jena, D-07743 Jena, Germany}
\email{m.keller@uni-jena.de}

\def\abstractname{Abstract}
\begin{abstract}
We introduce a curvature function for planar graphs to study the connection between the curvature and the geometric and spectral properties of the graph. We show that non-positive curvature implies that the graph is infinite and locally similar to a tessellation. We use this to extend several results known for tessellations to general planar graphs. For non-positive curvature, we show that the graph admits no cut locus and we give a description of the boundary structure of distance balls. For negative curvature, we prove that the interiors of minimal bigons are empty and derive explicit bounds for the growth of distance balls and Cheeger's constant. The latter are used to obtain lower bounds for the bottom of the spectrum of the discrete Laplace operator. Moreover, we give a characterization for triviality of essential spectrum by uniform decrease of the curvature. Finally, we show that non-positive curvature implies absence of finitely supported eigenfunctions for nearest neighbor operators.
\end{abstract}

\maketitle
\thispagestyle{empty}
\section{Introduction}\label{s:Intro}
There is a long tradition of studying planar graphs, i.e., graphs that can be embedded in a topological surface homeomorphic to $\R^2$ without self intersection. In particular, the curvature of  tessellating graphs received a great deal of attention during the recent years, see for instance \cite{BP1,BP2,CC,DM,Fo,H,K,KP,S1,SY,W,Z} and references therein. Here, we introduce a notion of curvature for general planar graphs to study their geometry and spectral properties. In the case of tessellations, which are also called tilings, the definition coincides with the one of \cite{BP1,BP2,H,S1,W}.

We next summarize and discuss the main results of the paper. Precise formulations can be found in the sections, where the results are proven.

The key insight of this paper is that non-positive curvature alone
has already very strong implications on the structure of a planar graph, i.e.:
\begin{itemize}
\item[(1.)] The graph is infinite and locally tessellating (Theorem~\ref{t:nonpos} in Section~\ref{s:NonPos}).
\item[(2.)] The graph is locally similar to a tessellation (Theorem~\ref{t:embedding} in Section~\ref{s:LocTessVsTess}).
\end{itemize}
Let us comment on these points. In contrast to the statement about the infinity, positive curvature implies finiteness of the graph. This was studied in \cite{CC,DM,H,S1,SY}.
Note also that locally tessellating graphs, recently introduced in \cite{KP} (see also \cite{Z}), allow for a unified treatment of planar tessellations and trees. The idea of (2.) is to construct an embedding  of a locally tessellating graph into a tessellation that preserves crucial properties of a fixed finite subset. This embedding allows us to carry over many  results known for tessellations to planar graphs. Indeed, most of the results are then direct corollaries of (1.), (2.) and the corresponding results for tessellations. In particular, we obtain for planar graphs with non-positive curvature:
\begin{itemize}
\item[(3.)] Absence of cut locus, i.e., every distance minimizing path can be continued to infinity (Theorem~\ref{t:no_cut_locus} in Section~\ref{s:no_cut_locus}).
\item[(4.)] A description of the boundary  of distance balls (Theorem~\ref{t:admissiblity} in Section~\ref{s:locstructure}).
\end{itemize}
For regular tessellations a result similar to (3.) was obtained by Baues, Peyerimhoff in \cite{BP1} and was later extended to general tessellations by the same authors in \cite{BP2}. This result is considered  a discrete analog of the Hadamard-Cartan theorem in differential geometry. Questions about the possible depth of the cut locus for Cayley graphs were studied under the name dead-end-depth in \cite{CR}, see also references therein. The boundary structure of distance balls, mentioned in (4.), plays a crucial role in the techniques of $\dot{\mbox{Z}}$uk, \cite{Z}. These concepts were later refined by \cite{BP1,BP2} under the name admissibility of distance balls. Here, we generalize the most important statements of \cite{BP2,Z} to planar graphs.

For negatively curved planar graphs, we prove the following:
\begin{itemize}
  \item[(5.)] Bounds for the growth of distance balls (Theorem~\ref{t:growthDistanceBalls} in Section~\ref{s:growthDistanceBalls}).
  \item[(6.)] Positivity and bounds for Cheeger's constant (Theorem~\ref{thm:cheegest} in Section~\ref{s:cheeger}).
  \item[(7.)] Empty interior for minimal bigons (Theorem~\ref{t:bigons} in Section~\ref{s:bigons}).
\end{itemize}
In \cite{BP1}, a result similar to (5.) is shown for tessellations. As for (6.), positivity of Cheeger's constant was proven by Woess \cite{W} for tessellations and simultaneously by $\dot{\mbox{Z}}$uk in a slightly more general context. Later, this was rediscovered in \cite{H}. Here, we give new explicit bounds for Cheeger's constant which can be interpreted in terms of the curvature. For related results, see for instance \cite{HJL,HiShi,K,KP}.
The results (6.) and (7.) can be understood as hyperbolic properties of a graph. In particular, if $G$ is the Cayley graph of a finitely generated group both results imply Gromov hyperbolicity. In general, this is not the case in our situation. However, under the additional assumption that the number of edges of finite polygons is uniformly bounded, empty interior for minimal bigons, along with Gromov hyperbolicity can be proven, see \cite{Z} and \cite{BP2}.

Let us now turn to the spectral implications. It is well known that bounds for Cheeger's constant imply bounds for the bottom of the spectrum of the graph Laplacian, see \cite{BMS-T,D,DKe,F,M}. The Cheeger constant also plays an important role in random walks. In particular, the simple random walk is transient if Cheeger's constant is positive, see \cite{G,W,W2} and also \cite{So}. Along with these well known relations, we show for planar graphs with non-positive curvature:
\begin{itemize}
  \item[(8.)] Triviality of the essential spectrum of the Laplacian if and only if the curvature decreases uniformly to $-\infty$ (Theorem~\ref{t:rap} in Section~\ref{s:specbounds}).
  \item[(9.)] Absence of finitely supported eigenfunctions for nearest neighbor operators (Theorem~\ref{t:no_cpt_supp_ef} in Section~\ref{s:no_cpt_supp_ef}).
\end{itemize}
Statement (8.) generalizes  results of \cite{K} for tessellations and of \cite{F} for trees. Statement (9.) is an extension of \cite{KLPS}. As for general planar graphs eigenfunctions of finite support can occur, this unique continuation statement shows that the analogy between Riemannian manifolds and  graphs is much stronger in the case of non-positive curvature.

The paper is structured as follows. In Section~\ref{s:Def}, we give the basic definitions. Section~\ref{s:NonPos} is devoted to the statement and the proof of Theorem~\ref{t:nonpos}. In Section~\ref{s:LocTessVsTess}, we construct the embedding into tessellations. %The applications of Theorem~\ref{t:nonpos} and the embedding are given in Sections~\ref{s:GeoApp} and~\ref{s:SpApp}.
Section~\ref{s:GeoApp} concentrates on the geometric applications (3.)-(7.) and Section~\ref{s:SpApp} is devoted to spectral applications (8.) and (9.).

\section{Definitions and a combinatorial Gauss-Bonnet formula}\label{s:Def}
Let $G=(V,E)$ be a graph embedded in an oriented topological surface $\mathcal{S}$. An embedding is a continuous one-to-one mapping from a topological realization of $G$ into $\mathcal{S}$. A graph admitting such an embedding into $\R^2$ is referred to as \emph{planar}. We will identify $G$ with its image in $\mathcal{S}$.
The faces $F$ of $G$ are defined as
the closures of the connected components of $\mathcal{S}\setminus \bigcup E$. We write $G=(V,E,F)$.

We call $G$ \emph{simple} if it has no loops (i.e., no edge contains only one vertex) and no multiple edges (i.e., two vertices  are not connected by more than one edge). We say that $G$ is \textit{locally finite} if for every point in $\mathcal{S}$ there exists an open neighborhood of this point that intersects with only finitely many edges. A characterization whether a Cayley graph of a group allows for a planar embedding that is locally finite was recently given in \cite{Geo}.

For the rest of this paper, we will assume that $G$ is locally finite.

We call two distinct vertices $v,w\in V$ \textit{adjacent} and we write $v\sim w$ if there is an edge containing both of them. We also say that two elements of $V$, $E$ and $F$ (possibly of distinct type) are \textit{adjacent} if their intersection is non-empty and we call two adjacent elements of the same type \textit{neighbors}.

By the embedding, each edge corresponds uniquely  to a curve (up to parametrization). We call a sequence of edges a \emph{walk} if the corresponding curves can be composed to a curve. Here, we allow for two sided infinite sequences. The \emph{length of the walk} is the number of edges in the sequence, whenever there are only finitely many and infinite otherwise. We refer to a subgraph, where the edges form a walk and each vertex is contained in exactly two edges as a \emph{path}.  The \emph{length of a path} is the length of the shortest walk passing all vertices. We call vertices that are contained in two edges of a walk or a path the \emph{inner vertices} of the walk or the path. A walk or a path is called \emph{closed} if every vertex  contained in it is an inner vertex and \emph{simply closed} if all vertices are contained in exactly two edges. Note that every simply closed walk induces a path and every closed path is simply closed.
We say a graph is \emph{connected} if any two vertices can be joined by a path.

For the remainder of this paper, we will assume  that $G$ is connected.

For a face $f\in F$, we call a walk a \emph{boundary walk} of $f$ if it meets all vertices in $f$ and  if it is closed whenever it is finite. %and if there are curves in the interior of $f$  that are arbitrary close and homotopic in $f$ to the curve corresponding to the walk.
The existence of boundary walks for all faces is ensured by the connectedness of the graph.
We let the \emph{degree} $|f|$ of a face $f\in F$ be the length of the shortest boundary walk of $f$, whenever $f$ admits a finite boundary walk and infinite otherwise.  We define the \emph{degree} $|e|$ of an edge $e\in E$ as the number of vertices contained in $e$. For a vertex $v\in V$, we define the \emph{degree} by $$|v|:=2\sum_{e\in E,v\in e}\frac{1}{|e|}.$$
The formula can be interpreted as the number of adjacent edges with degree two plus twice the number of adjacent edges with degree one. The latter are counted twice since they meet $v$ twice. A vertex $v\in V$ with $|v|=1$ is called a \emph{terminal vertex}.

To define curvature functions, we first have to introduce the
corners of a planar graph $G$. The set of \emph{corners} $C(G)$ is a subset of $V\times F$ such that the elements $(v,f)$ satisfy $v\in f$. We denote for a vertex $v$  the set   $C_{v}(G):=\{(v,g)\in C(G)\}$ and for a face $f$ the set $C_{f}(G):=\{(w,f)\in C(G)\}$. The \emph{degree} or the \emph{multiplicity} $|(v,f)|$ of a corner $(v,f)\in C(G)$ is the minimal number of times the vertex $v$ is met by a boundary path of $f$. For tessellations, the definition of corners coincides with the one in \cite{BP1,BP2} and there every corner has degree one. However, if $G$ is not a tessellation the degree of a corner can be larger than one, (see for instance the left hand side of Figure~\ref{f:prototypes} in the next section). For a vertex $v\in V$ and a face $f\in F$, we have
\begin{align*}
    |v|=\sum_{(v,g)\in C_{v}(G)}|(v,g)|
\qquad\mbox{and}\qquad    |f|=\sum_{(w,f)\in C_{f}(G)}|(w,f)|.
\end{align*}
We can think of the corners of a vertex $v$ with respect to the multiplicity as the partitions of a sufficiently small ball after removing the edges adjacent to $v$, (where, small means that no edge is completely included in the ball and there are exactly $|v|$ partitions).

The \emph{curvature function on the corners}
$\ka_C^G:C(G)\rightarrow\R$ is defined by
$$\ka_C^G(v,f):=\frac{1}{|v|}-\frac{1}{2}+\frac{1}{|f|},$$
with the convention that $1/|f|=0$ whenever $|f|=\infty$.
We define the \emph{curvature function on the vertices}
$\ka_V^G:V\rightarrow\R$ by
$$\ka_V^G(v):=\sum_{(v,f)\in C_v(G)}|(v,f)|\ka^G_C(v,f).
$$
By direct calculation, we arrive at
\begin{align*}
\ka_V^G(v)=1-\frac{|v|}{2}+\sum_{(v,f)\in C_v(G)} |(v,f)| \frac{1}{|f|}.
\end{align*}
We define the \emph{curvature function on the faces}
$\ka^G_F:F\rightarrow [-\infty,\infty)$ by
$$\ka^G_F(f):=\sum_{(v,f)\in C_f(G)}|(v,f)|\ka^G_C(v,f).$$
If $|f|=\infty$ and if there are infinitely many vertices in $f$ with vertex degree of at least three, then
the curvature of a face $f$  takes the value $-\infty$.

For finite subsets $V'\subseteq V$, $F'\subseteq F$, we write
$$\ka^G_V(V'):=\sum_{v\in V'}\ka_V^{G}(v)
\qquad\mathrm{and}\qquad\ka^G_F(F'):=\sum_{f\in F'}\ka^G_F(f).$$
Moreover, we let
$$\ka_C(G):=\sup_{(v,f)\in C(G)}\ka_C^G(v,f),\quad \ka_V(G):=\sup_{v\in V}\ka_V^G(v),\quad
\ka_F(G):=\sup_{f\in F}\ka_F^G(f).$$
%\Hm{Does it suffice to ask that the boundary is a path} NO! circle with line to the center
We call a face a \emph{polygon} if it is homeomorphic to the closure of the unit disc $\D$ in $\R^2$ and its boundary is a closed path. We call a face an \emph{infinigon} if it is homeomorphic to $\R^2\setminus \D$ or the upper half plane $\R\times\R_+\subset \R^2$ and its boundary is a path. A graph $G$ is called \emph{tessellating} if the following conditions are satisfied:
\begin{itemize}
\item [(T1)] Every edge is contained in precisely two different faces.\label{T2*}
\item [(T2)] Every two faces are either disjoint or intersect precisely
in a vertex or in an edge.
\item [(T3)] Every face is a polygon.
\end{itemize}
The additional assumption in \cite{BP1,BP2} that each vertex has finite degree is already implied by local finiteness of $G$. Note that the dual graph of a non-positively corner curved tessellation is also a non-positively corner curved tessellation.  For a discussion, we refer to \cite{BP1}.

We call a path of at least two edges an \emph{extended edge} if all inner vertices have degree two and the beginning and ending vertex, in the case that they exist, have degree greater than two. An extended edge contained in two faces with infinite degree is called \emph{regular}. We introduce two  weaker conditions  than (T2) and (T3):
\begin{itemize}
\item [(T2*)] Every two faces are either disjoint or intersect precisely in a vertex, an edge or a regular extended edge.
\item [(T3*)] Every face is a polygon or an infinigon.
\end{itemize}
We call a graph satisfying (T1), (T2), (T3*) \emph{strictly locally tessellating} and a graph
satisfying (T1), (T2*), (T3*) \emph{locally tessellating}. Note that a locally tessellating graph that contains no extended edge is strictly locally tessellating. While the graphs studied in \cite{BP1,BP2,CC,DM,H,HiShi,K,KLPS,So,S1,SY,W} are tessellating, the results of \cite{KP,Z} concern strictly locally tessellating graphs. We give some examples.

\begin{Example}
(1.) Tessellations of the plane $\R^2$
%, the half-plane $\R\times\R_+$
or the sphere $\Sp^2$ are strictly locally tessellating. Furthermore, tessellations of $\Sp^2$ embedded into $\R^2$ via stereographic projection are also strictly locally tessellating.\\
(2.) Trees are strictly locally tessellating if and only if the branching number of every vertex is greater than one. In this case, the tree is negatively curved in each corner, vertex and face. \\
(3.) The graph with vertex set $\Z$ and  edge set $\{[n,n+1]\}_{n\in\Z}$ is locally tessellating, but not strictly locally tessellating and has curvature zero in every corner, vertex and face.
\end{Example}
Next, we will prove a combinatorial Gauss-Bonnet formula.
We refer to \cite{BP1} for  background and
proof in the case of tessellations, see also \cite{CC,DM} for further reference.

Let $G=(V,E,F)$ be a planar graph embedded into $\Sp^2$ or $\R^2$. For a subset $W\subseteq V$, we denote by $G_W=(W,E_W,F_W)$ the subgraph of $G$ induced by the vertex set $W$, where $E_W\subseteq E$ are the edges that contain only vertices in $W$ and $F_W$ are the faces induced by the graph $(W,E_W)$.
For a finite connected subset of vertices $W\subseteq V$,  Euler's formula reads
\begin{equation*}%\label{e:Euler\sum}
|V_W|-|E_W|+|F_W|=2.
\end{equation*}
Observe that $F_W$ contains also an unbounded face which explains the two on the right hand side.

\begin{Proposition}\textbf{(Gauss-Bonnet formula)}
\emph{Let $G$ be planar and $W\subseteq V$ finite and connected. Then $$\ka^{G_W}_V(W)=2.$$}
\begin{proof}
We have, by definition,
\begin{eqnarray*}
\ka^{G_W}_V(W)   &=&  \sum_{(v,f)\in
C(G_W)}|(v,f)|\ab{\frac{1}{|v|}+\frac{1}{|f|}-\frac{1}{2}}\\
   &=&\sum_{v\in
W}\sum_{c\in C_v(G_W)}\frac{|c|}{|v|}+\sum_{f\in F_W}\sum_{c\in
C_f(G_W)}\frac{|c|}{|f|}-\sum_{v\in W}\sum_{c\in C_v(G_W)}\frac{|c|}{2}.
\end{eqnarray*}
Since $|v|=\sum_{c\in C_v(G_W)}|c|$, the first term is equal to $|W|$.
As $W$ is connected, we  have $|f|=\sum_{c\in C_f(G_W)}|c|$. Thus, the second term is equal to $|F_W|$. Let $E_{v,j}=\{e\in E_W\mid v\in e,\;|e|=j\}$ for $v\in W$ and $j=1,2$. Then, $\sum_{c\in C_v(G_W)}|c|=|v|=2|E_{v,1}|+|E_{v,2}|$ by the definition of $|v|$. Moreover, for each $e\in E_{v,2}$ there is a unique $w\in W$, $w\neq v$, such that $e\in E_{w,2}$. We conclude
\begin{eqnarray*}
\sum_{v\in W}\sum_{c\in C_v(G_W)}\frac{|c|}{2} = \sum_{v\in
W}|E_{v,1}|+\frac{1}{2}|E_{v,2}|=|E_W|.
\end{eqnarray*}
By Euler's formula, we obtain the result.
\end{proof}
\end{Proposition}

\begin{Remark}
The Gauss-Bonnet formula immediately implies that a finite graph must admit some positive curvature. This is, in particular, the case for locally finite planar graphs embedded in $\Sp^2$.
\end{Remark}

%We end the section with some definitions.
We write $d(v,w)$ for the length of the shortest path connecting the vertices $v$ and $w$.
For a set of vertices $W\subseteq V$, we define the balls and the spheres of radius $n$ by
\begin{align*}
  B_n(W) &:=  B_n^G(W):=\{v\in V\mid d(v,w)\leq n\mbox{ for some } w\in W\},\\
  S_n(W) &:=  S_n^G(W):=B_n(W)\setminus B_{n-1}(W).
\end{align*}
We define the  \emph{boundary faces} of $W$ by
$$\dd_F W:=\dd^G_F W:=\{f\in F\mid f\cap  W\neq\emptyset,\;\;  f\cap V\setminus W\neq\emptyset\}.$$
%We let the set of boundary edges of a face $f\in F$ be given by
%$$\dd_E f:=\dd^G_E f:=\{e\in E\mid e\in f\}.$$

\section{Infinity and local tessellating properties of non-positively curved planar graphs}\label{s:NonPos}
This section is dedicated to prove the first main result.
\begin{Thm}\label{t:nonpos}
\emph{Let $G$ be a  planar graph that is connected and locally finite.
If one of the following conditions is satisfied
\begin{itemize}
  \item [(a.)] $\ka_C(G)\leq 0$ or
  \item [(b.)] $\ka_V(G)\leq 0$ and $G$ is simple or
  \item [(c.)] $\ka_F(G)\leq0$, each extended edge
  is regular and there are no terminal vertices,
\end{itemize}
then  $G$ is infinite and locally tessellating. If a strict
inequality holds in condition $\mathrm{(a.)}$ or $\mathrm{(b.)}$ then $G$ is even strictly locally tessellating.}
\end{Thm}

The main idea of the proof can be summarized as follows. The infinity is a direct consequence of the Gauss-Bonnet formula. Moreover, assuming a graph is not locally tessellating, we will construct a finite planar graph with smaller or equal curvature which can be embedded into $\Sp^2$. By the Gauss-Bonnet formula, this graph must admit some positive curvature. As the curvature is smaller or equal, the original graph must have some positive curvature as well.

We start with analyzing the pathologies that occur in general planar graphs that are not locally tessellating. Let $G=(V,F,E)$ be a planar graph that is connected and locally finite. A face $f\in F$ is called \emph{degenerate} if it contains a vertex $v$ such that $|(v,f)|\geq2$. Note that this is in particular the case if $v$  is contained in three or more boundary edges  $ f$ or  $f$ includes and an edge that is included in no other face than $f$. %In this case, we also call the vertex $v$ and the edge $e$ \emph{degenerate}.
A pair of faces $(f,g)$ is called \emph{degenerate} if $f\cap g$ consists of at least two
connected components. Figure \ref{f:prototypes} shows examples of degenerate faces. We denote by $D(F)\subseteq F$ the set of all faces that are degenerate or are contained in a degenerate pair of faces.
\begin{figure}[!h]
\scalebox{0.3}{\includegraphics{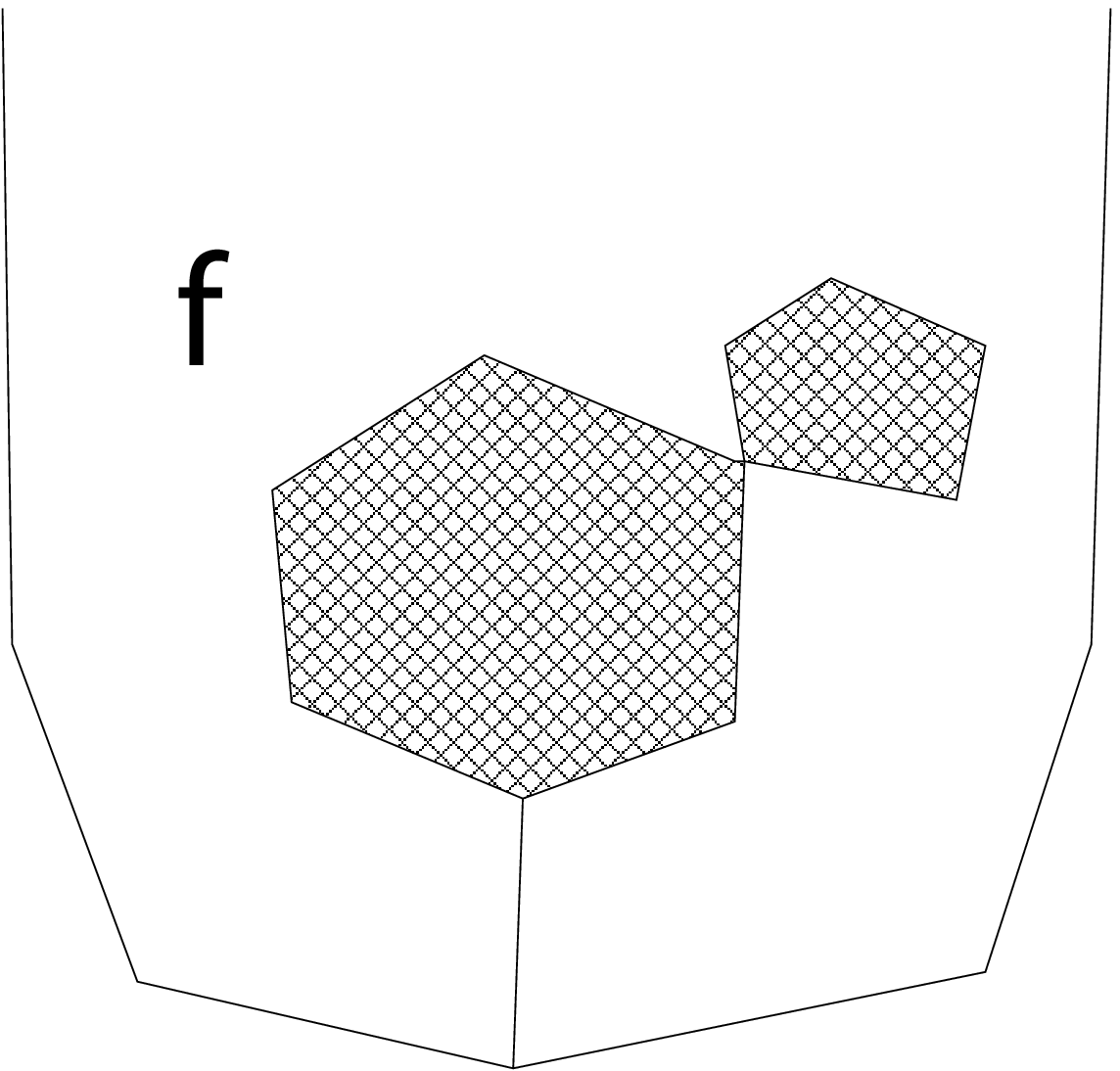}}
\hspace{1cm}\scalebox{0.3}{\includegraphics{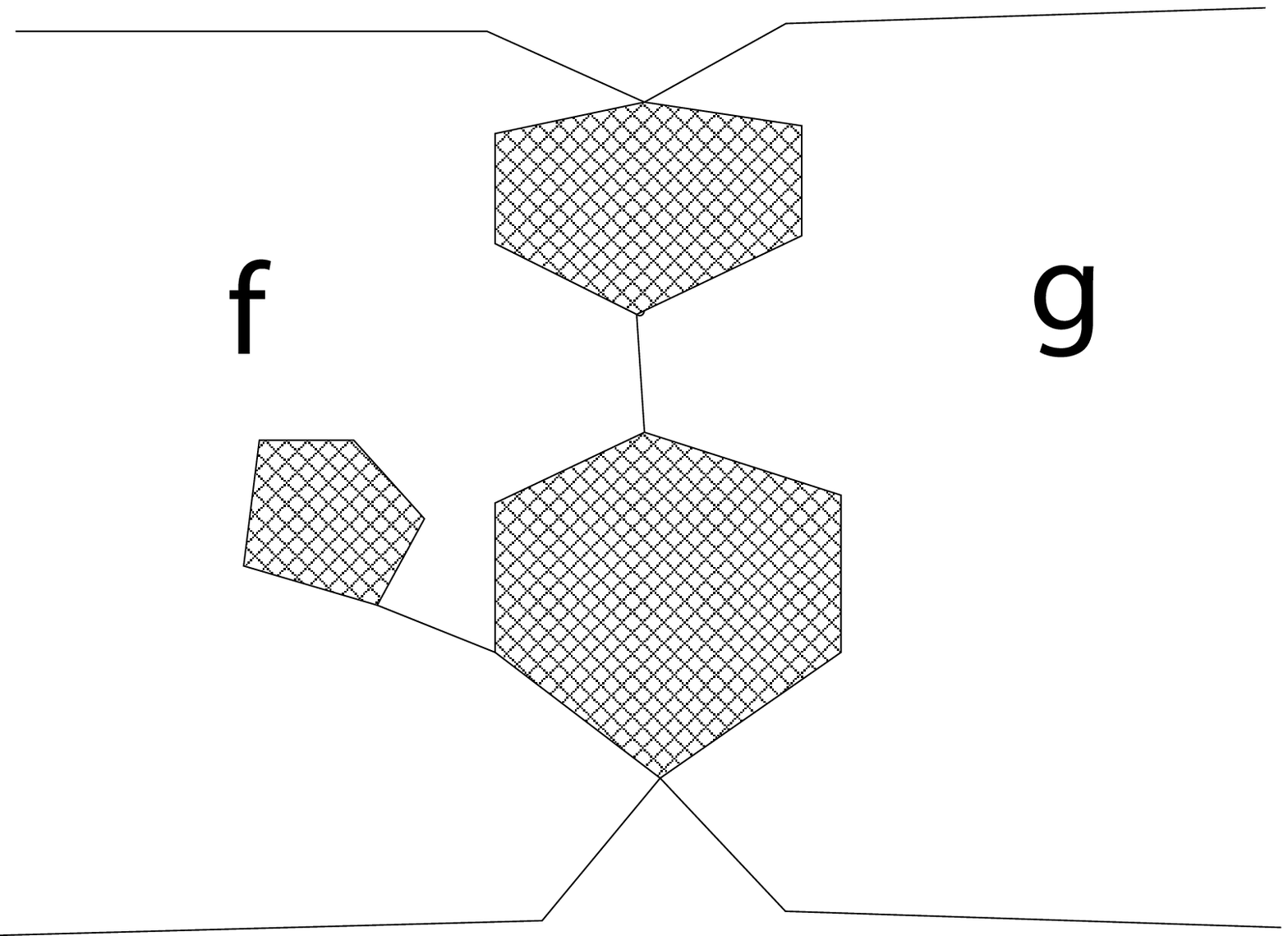}}
\caption{\label{f:prototypes}\scriptsize Examples of a degenerate face $f$ and a degenerate pair of
faces $(f,g)$.}
\end{figure}
%Moreover, we denote by $D(V)$ the set of all degenerate vertices, all vertices contained in a degenerate edge or in the intersection of a degenerate pair of faces.

\begin{Lemma}\label{l:nonpos_subgraph}
\emph{Let $G=(V,E,F)$ be a  simple, planar graph that is connected,   locally finite and contains no terminal vertices. If $D(F)\neq \emptyset$, then there is a finite and connected subgraph $G_1=(V_1,E_1)$ which is bounded by a simply closed path and satisfies the following property: there are at most two vertices $v_1,v_2\in V_1$ such that each path connecting $V_1$ and $V\setminus V_1$ meets either $v_1$ or $v_2$. In particular, $v_1$, $v_2$ lie in the boundary path of $G_1$.}
\begin{proof}
The proof consists of two steps. Firstly, we find a finite set $W\subseteq V$ that contains at most two vertices with the asserted property. We have to deal with the case of a degenerate face and a degenerate pair separately.  Secondly, we find a subgraph of $G_W$ that has a simply closed boundary.
\begin{figure}[!h]
\scalebox{0.37}{\includegraphics{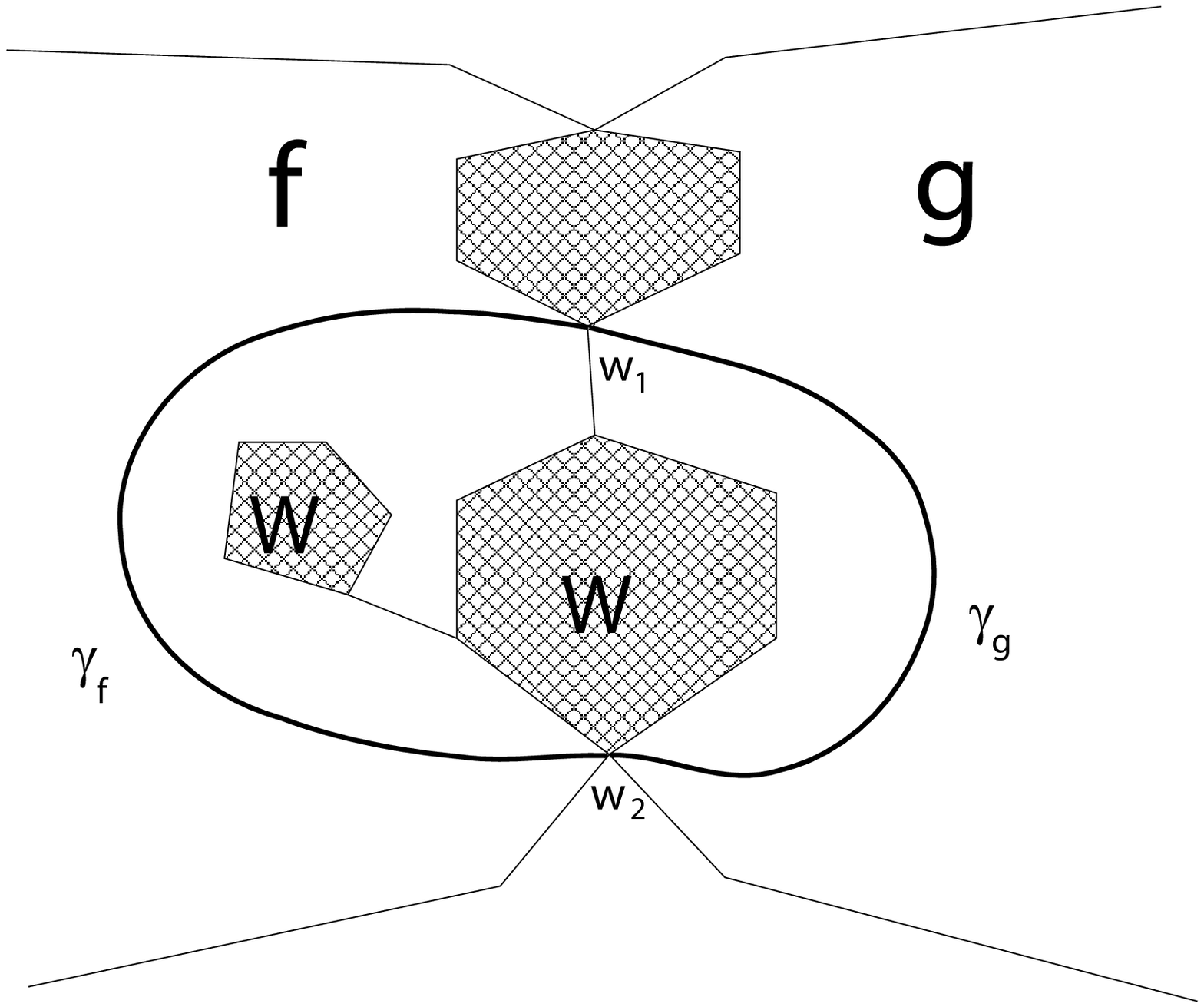}}
\scalebox{0.37}{\includegraphics{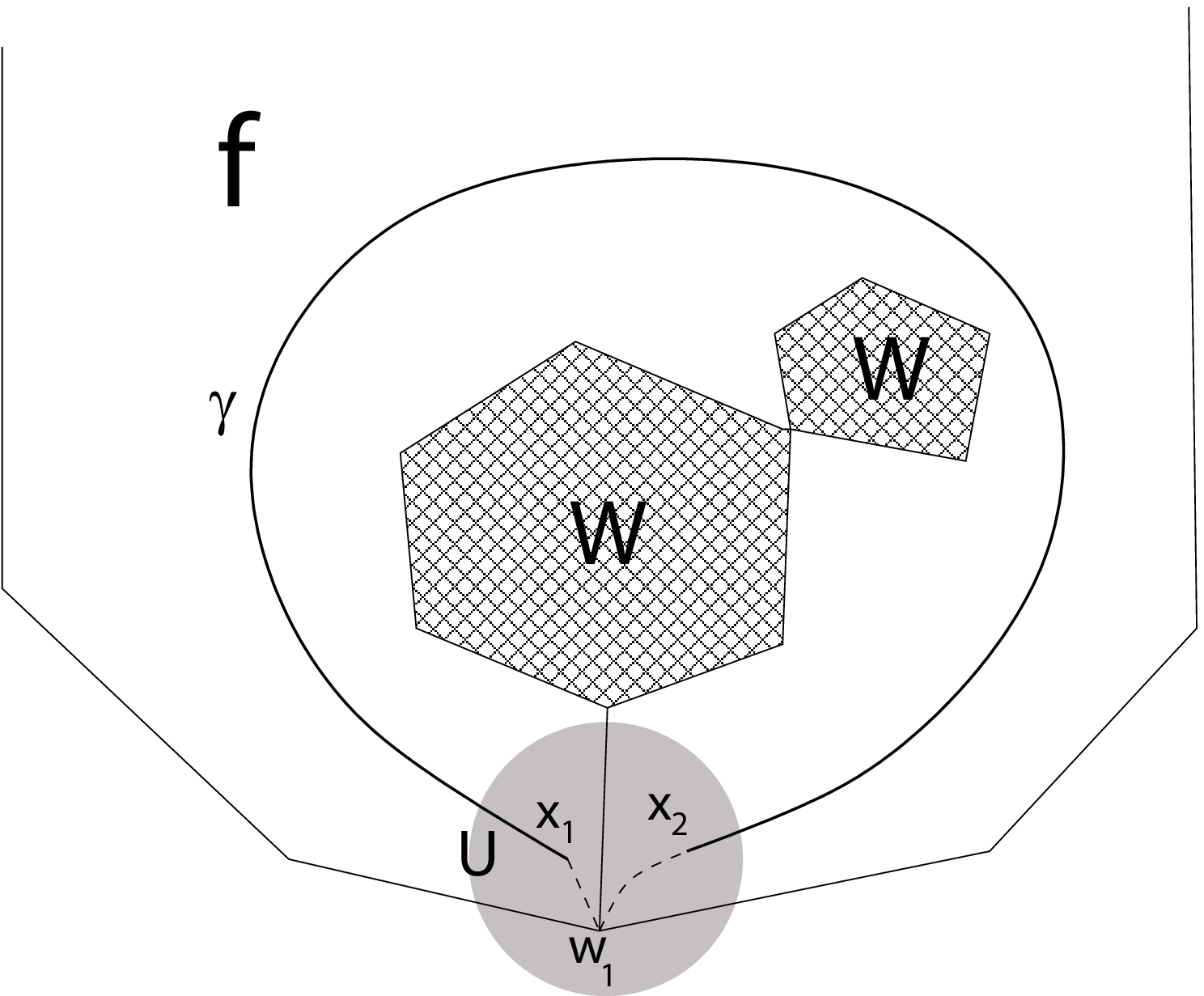}}
\caption{\label{f:W}\scriptsize  The figure shows how to isolate a finite subset $W$, whenever there is degenerate pair of faces or a degenerate face.}
\end{figure}
\\
We start with the case of a degenerate pair of faces $(f,g)$. Let
$w_1,w_2$ be two vertices that are contained in different connected components of the intersection of $f$ and $g$. Let $\gm_f$ be a simple curve that lies in $(\inn\, f)\cup\{w_1,w_2\}$ and connects $w_1$ and $w_2$, where
$\inn\, f:=f\setminus\bigcup\{e\in E\mid e\in f\}$. Similarly, let $\gm_g$ be a simple curve in $(\inn\, g)\cup\{w_1,w_2\}$ connecting $w_1$ and $w_2$. Composing these curves, we obtain a simply closed curve $\gm$ which
divides the plane by Jordan's curve theorem into a bounded and an unbounded component. We denote the bounded component by $B$. The set $W=V\cap B$ can be connected to $V\setminus W$ only by walks which meet $w_1$ or $w_2$. For an illustration see the left hand side of Figure~\ref{f:W}.\\
We now turn to the case when there is no degenerate pair of faces. As $D(F)\neq\emptyset$, there must be a degenerate face $f$. Let $w_1$ be a vertex in $f$ with $|(w_1,f)|\geq2$. We pick an open simply connected neighborhood $U$ of $w_1$ that contains no other vertex except for $w_1$ and $e\cap U$ is connected for all edges $e$ adjacent to $w_1$. Let $x_1$ and $x_2$ be two arbitrary points in two different connected components of $(\inn\, f)\cap U$. We connect $x_1$ and $x_2$ by a simple curve in $\inn\, f$ and connect $x_1,x_2$ and $w_1$ also by simple curves that lie in the corresponding connected component of $U$. By composition, we obtain a simply closed curved and by Jordan's curve theorem we get a bounded set $B$. Note that $W=V\cap B$ can be connected
to $V\setminus W$ only by walks which  meet $w_1$. For an illustration see the right hand side of Figure~\ref{f:W}.\\
In both cases, we identified a set $W$ which is finite and connected since the graph is locally finite and connected. Now, we find a subset $V_1\subseteq W$ such that
$G_1=(V_1,E_1)$ is bounded by a simply closed path. By construction,  $w_1$ lies in an unbounded face of $F_W$, which we denote by $f$ and which has a finite boundary walk. We start from $w_1$ walking around $f$. The walk might visit certain vertices several times. The vertex $w_1$ is visited at least twice, since the walk is finite. We pick a subsequence of the walk such that no vertex is met twice except for the starting and ending vertex, which we denote by $v_1$. In Figure~\ref{f:bdry}, this is illustrated (with $v_1=w_3$).\begin{figure}[!h]
\scalebox{0.4}{\includegraphics{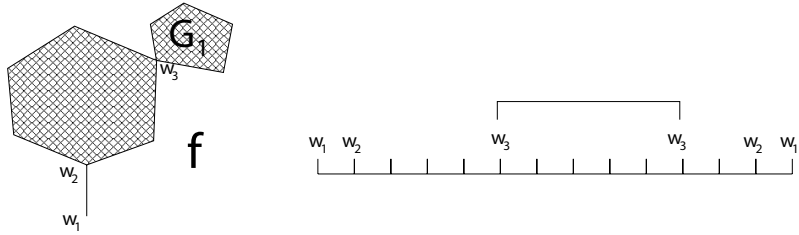}}
\caption{\label{f:bdry}\scriptsize The right hand side shows an enumeration of the boundary walk around $G_W$ illustrated on the left hand side. To isolate the subgraph $G_1$, we pick the simply closed
subwalk from $w_3$ to $w_3$.}
\end{figure}\\
The sequence of edges forms a closed path and
encloses a subgraph, which we denote by $G_1=(V_1,E_1)$. Note that $V_1\subseteq V$ contains $v_1$ which might be equal to $w_1$ and $V_1$ also might contain $w_2$ which we denote in this case by $v_2$. Nevertheless, $v_1$ and $v_2$ are the only vertices in $V_1$ that can be connected to $V\setminus V_1$ by edges of $G$. Thus, $G_1$ has the desired properties and we finished the proof.
\end{proof}
\end{Lemma}

The next lemma is the main tool for the proof of Theorem \ref{t:nonpos}. It shows that the existence of
degenerate faces or pairs implies positive curvature.

\begin{Lemma}\textbf{(Copy and paste lemma)} \label{l:non-pos_CopyPaste}
\emph{Let $G=(V,E,F)$ be a  simple, planar graph that is connected,   locally finite, contains no terminal vertices and all extended edges are regular. If $D(F)\neq \emptyset$, then there are $F'\subseteq F$ and $V'\subseteq V$ such that $\ka^G_F(F')> 0$ and $\ka^G_V(V')> 0$.}
\begin{proof}
The idea of the proof is easy to illustrate. We assume $D(F)\neq \emptyset$ and take the subgraph $G_1$ of Lemma~\ref{l:nonpos_subgraph}. We make copies of $G_1$, paste them along the boundary and embed the resulting graph into $\Sp^2$. See Figure~\ref{f:CopyPaste} below. We then show that the curvature of this graph compared to the curvature in $G$ does not increase, as long as we make enough copies of $G_1$. The statement is then implied by the Gauss-Bonnet formula.
\begin{figure}[h]
\scalebox{0.48}{\includegraphics{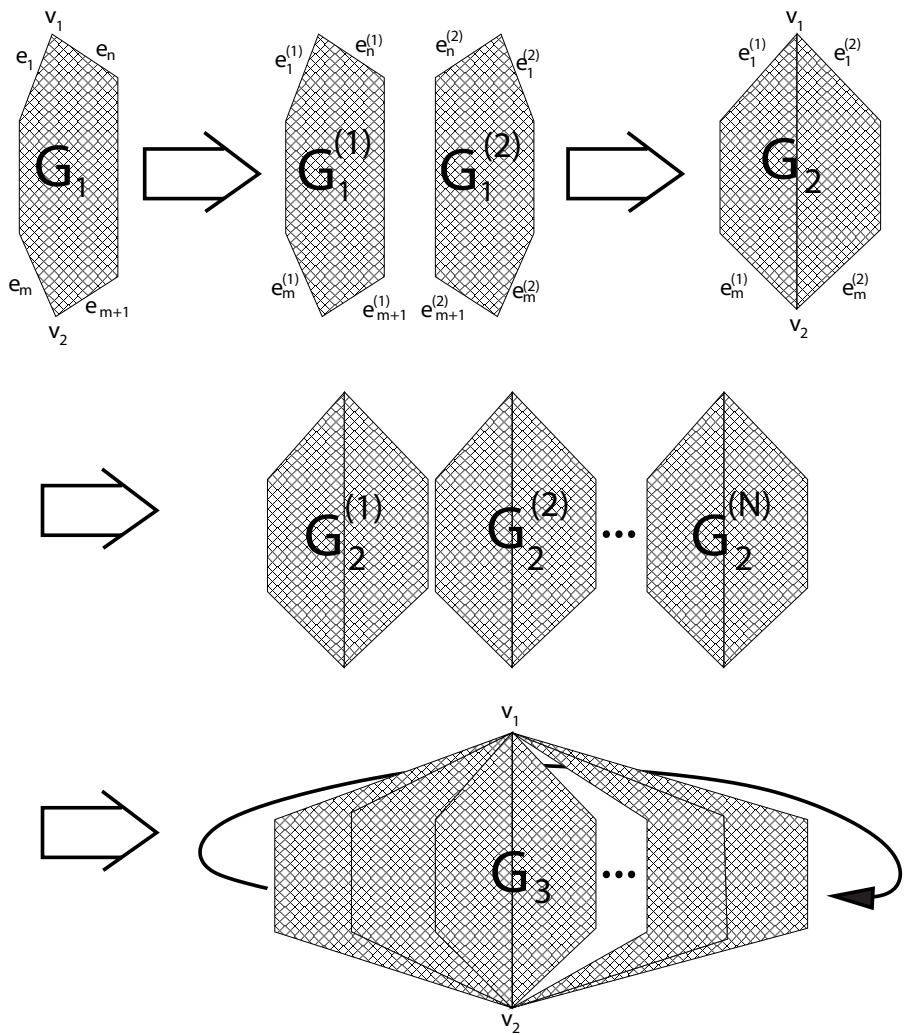}} \caption{\label{f:CopyPaste} An
illustration of the copy and paste procedure.}
\end{figure}
\\
Assume $D(F)\neq \emptyset$. By Lemma~\ref{l:nonpos_subgraph}, there is a finite subgraph $G_1=(V_1,E_1,F_1)$ of $G$ which is enclosed
by a closed path $p$. The vertex degree in $G_1$ differs from $G$ in at most two vertices, which we denote by $v_1$ and $v_2$. (If there is only one vertex, then we choose another vertex arbitrarily in the boundary path of $G_1$.) Let $\{e_1,\ldots,e_n\}$ be the
edges of the boundary path of $G_1$ starting and ending at $v_1$, i.e., $v_1=e_1\cap e_n$. Moreover, let $m<n$ such that $v_2=e_m\cap e_{m+1}$. Denote $p_1=\{e_1,\ldots,e_m\}$ and
$p_2=\{e_{m+1},\ldots,e_n\}$. We take two copies
$G_1^{(1)}$ and $G_1^{(2)}$ of $G_1$. We paste them along the edges of the subpaths $p_2^{(1)}$ and $p_2^{(2)}$ (where $p_2^{(j)}$ corresponds to $p_2$ in $G_1^{(j)}$ for $j=1,2$), i.e., we identify the edges
$\{e_{m+1}^{(1)},\ldots,e_n^{(1)}\}$ in $G_1^{(1)}$ with the edges $\{e_{m+1}^{(2)},\ldots,e_n^{(2)}\}$ in $G_1^{(2)}$. We denote the resulting graph by $G_2$. Note that the edges of the boundary path of $G_2$ are the set $p_1^{(1)}\cup p_1^{(2)} =\{e_1^{(1)},\ldots,e_m^{(1)},e_1^{(2)},\ldots,e_m^{(2)}\}$.
Denote $q_1=p_1^{(1)}$ and $q_2=p_1^{(2)}$. Now, let $N$ be an integer, that will be quantified later. We take $N$ copies $G_2^{(1)},\ldots,G_2^{(N)}$ of $G_2$ and paste them along the subpaths $q_2^{(j)}$ and $q_1^{(j+1)}$ for $j=1,\ldots ,N-1$.
We embed the resulting planar graph into $\Sp^2$ and we finally paste $q_1^{(1)}$ and $q_2^{(N)}$. We denote the resulting graph embedded in $\Sp^2$ by $G_3=(V_3,E_3,F_3)$. In Figure \ref{f:CopyPaste}, the procedure is illustrated. \\
We will make the following observations which are implied by the construction:
\begin{itemize}
\item [(O1)] We can identify each corner in $C(G_3)$ uniquely with a corner in $C(G)$ and $C(G_1)$ and they have the same multiplicity. For each such corner in $C(G)$ or $C(G_1)$ there are exactly $2N$ corresponding corners in $C(G_3)$.
\item [(O2)] We can identify each face in $F_3$ uniquely with a face in $F$ and $F_1$. Each bounded face in $F_1$ can be identified uniquely with a face in $F$ and exactly $2N$ corresponding faces in $F_3$. Moreover, the face degree of corresponding faces is the same in $G$, $G_1$ and $G_3$
\end{itemize}
We now quantify $N$ which was introduced above. Let
$$N=\max_{j=1,2} \left\lceil
\frac{3\av{v_j}}{2\ab{\av{v_j}_1-1}}\right\rceil,$$
where $\lceil x\rceil$ denotes the smallest integer that is greater than $x$ for $x\in \Q$ and $\av{\cdot}_1$ denotes the vertex degree in $G_1$. There is one more observation:
\begin{itemize}
\item [(O3)] We can identify each vertex in $V_3$ uniquely with a vertex in $V$ and in $V_1$. Each of these vertices in $V$ and $V_1$  either corresponds  to $v_1$, $v_2$ or to exactly $N$ vertices of the former boundary path $p$ of $G_1$ or to exactly $2N$ vertices in $V_3$. Moreover, the vertex degree of a vertex in $G_3$ is at least the vertex degree of the corresponding vertex in $G$ and it is at least three.
\end{itemize}
The statement about the corresponding vertices in (O3) follows by construction. To check the statement about the vertex degrees in (O3), one has to consider three cases. The statement is clear for vertices in $V_3$ that are not contained in the boundary path of $G_1$. For the vertices $v_1$, $v_2$ the statement follows by the choice of $N$. For any other vertex $v\in V_3$,  the vertex degree ${|v|}_3$ in $G_3$ is equal to $\sum_{c\in C_v(G_3)}|c|=2(|v|-1)$, which yields the statement. Finally, we have ${|v|}_{3}\geq3$ for all $v\in V_3$ since we assumed that all extended edges in $G$ are regular and all faces contained in $F_3$ are bounded.\\
In the following, we will not distinguish between the objects we identified in (O1), (O2), (O3).  An obvious consequence of (O1), (O2) and (O3) is that for a corner $c\in C(G_3)$ and the corresponding corner $c\in C(G)$ we have
$$\ka_C^{G_3}(c)\leq \ka^G_C(c).$$

We now find  $F'\subseteq F$ that satisfies $\ka^{G}_F(F')>0$.
Denote the unbounded face in
$F_1$ by $f_\infty$ and set $F'=F_1\setminus\{f_\infty\}$ which we can identify with a subset of $F$ and $2N$ identical subsets of $F_3$ by (O2). Thus, we conclude by the Gauss-Bonnet formula and the inequality above
\begin{align*}
2&=\ka_{V}^{G_3}(V_3)=\sum_{c\in C(G_3)}|c|\ka_C^{G_3}(c)
=2N\sum_{f\in F'}\sum_{c\in C_f(G_3)}|c|\ka_C^{G_3}(c)\\
&\leq 2N\sum_{f\in F'}\sum_{c\in C_f(G)}|c|\ka_C^G\ab{c} =2N\sum_{f\in F'}\ka_F^G\ab{f}.
\end{align*}
Therefore, $\ka_F^G(F')\geq {1}/{N}>0$.

To find $V'\subseteq V$ with $\ka^{G}_V(V')>0$, we assume that all vertices $v\in V$ that correspond to a vertex in the boundary path of $G_1$ satisfy $\ka^{G}_V(v)\leq0$. Otherwise, we have $\ka_V^G(V')>0$ for $V'=\{v\}$.
We first show $\ka^{G_3}_V(v)\leq\ka_V^{G}(v)$ for $v\in V_{3}$.
To do so, we check three cases:
Firstly, the statement is clear for vertices $v$ that are not contained in the boundary path $p$ of $G_1$.
Secondly,
for $v\in V_3\setminus\{v_1,v_2\}$ that is contained in the boundary path of $G_1$, we have  ${|v|}_{3}=2(|v|-1)$. Hence,
\begin{align*}
\ka^{G_3}_V(v)&=1-\frac{{|v|}_{3}}{2}+ \hspace{-.cm}
\sum_{(v,f)\in C_v(G_3)} \hspace{-.cm} |(v,f)|\frac{1}{|f|}\\
&=2-|v|+\hspace{-.cm}
\sum_{
(v,f)\in C_v(G)}
\hspace{-.cm}
2|(v,f)|\frac{1}{|f|}-\frac{2}{|g|}
=2\ka_{V}^{G}(v)-\frac{2}{|g|}\leq\ka_{V}^{G}(v).
\end{align*}
where $g\in F$ is the unique face adjacent to $v$ in $G$ for which there is no corresponding face in $G_3$. Note that $g$ corresponds to $f_\infty $ in $G_1$ and $|(v,g)|=1$.  The second equality follows since for any other corner in $C_v(G)$ there are exactly two corresponding corners with the same multiplicity in $C_v(G_3)$. The last inequality is due to the assumption $\ka_{V}^{G}(v)\leq0$ for vertices $v$ contained in $p$.
Thirdly, we have to check  $\ka^{G_3}_V(v_j)\leq\ka^G_V(v_j)$ for $j=1,2$. In the following inequality, we estimate $|f|\geq3$, then use $\sum_{c\in C_{v_j}(G_{3})}|c|={|v_j|}_{3}=2N({|v_j|}_1-1)\geq3|v_j|$ from the definition of $N$, to obtain
\begin{align*}
\ka^{G_3}_V(v_j)=1-\frac{{|v_{j}|}_{3}}{2}+ \hspace{-.4cm}
\sum_{(v_{j},f)\in C_{v_{j}}(G_3)} \hspace{-.5cm} |(v_{j},f)|\frac{1}{|f|}\leq 1-\frac{{|v_j|}_{3}}{6}\leq 1-\frac{|v_j|}{2}.
\end{align*}
As $\ka^G_V(v_j)\geq1-{|v_j|}/{2}$, we have $\ka^G_V(v_j)\geq\ka^{G_3}_V(v_j)$ for $j=1,2$. Thus, we have shown for all $v\in V_3$
$$\ka^{G_3}_V(v)\leq\ka_V^{G}(v).$$
We finish the proof by identifying a subset $V'\subseteq V_3$ that satisfies $\ka^{G_3}_V(V')>0$ and, by the identification (O3), $\ka^{G}_V(V')>0$.
Let $V'\subset V$ be the vertices of $G_1$ that are not in the vertex set $V_p$ of the boundary path $p$ of $G_1$.
The Gauss-Bonnet formula and (O3) yield, since we assumed $\ka_{V}^{G}(v)\leq 0$ for $v\in V_p$
$$2=\ka^{G_3}_V(V_3)
=2N\sum_{v\in V'} \ka^{G_3}_V(v)+N\sum_{v\in V_p} \ka^{G_3}_V(v)+\sum_{j=1,2}\ka^{G_3}_V(v_j)\leq 2N\ka^{G}_V(V').$$
Therefore, $\ka_V^G(V')\geq\frac{1}{N}>0$. If $V'\neq\emptyset$, then we are done. Otherwise, we arrived at a contradiction to the assumption $\ka_{V}^{G}(v)\leq0$ for $v\in V_p$. Hence,  we conclude $\ka_{V}^{G}(V'')>0$ for $V''=\{v\}$ for the vertex $v$ that gives the contradiction.
\end{proof}
\end{Lemma}

The next lemma shows that the absence of degenerate faces and degenerate pairs characterizes whether a ''nice`` graph is locally tessellating.
\begin{Lemma} \label{l:non-pos_subtess_DF}
\emph{Let $G$ be a simple,  planar graph that is connected,  locally finite, contains no terminal vertices and each extended edge is regular. Then, $G$ is locally tessellating if and only if $D(F)=\emptyset$. In this case, each corner has multiplicity one and
$$\ka^G_V(v)=1-\frac{|v|}{2}+\sum_{f\in F, v\in f} \frac{1}{|f|}\qquad\mbox{for all $v\in V$}.$$}
\begin{proof}Obviously, if $D(F)$ is non-empty, then at  least one of the conditions (T1), (T2*), (T3*)  on page~\pageref{T2*} is violated. On the other hand, the absence of degenerate faces implies that each corner has multiplicity one. This has the following consequences: Firstly, no edge can be included in only one face which is (T1). Secondly, the boundary of each face is a path (i.e., there is a boundary walk of the face which is meeting every vertex only once) and, hence, (T3*) follows. Thirdly, the formula for the curvature now follows  from the definition. Finally, the absence of degenerate pairs of faces implies (T2*) as extended edges are assumed to be regular.
\end{proof}
\end{Lemma}
%\begin{Remark} The property $D(F)=\emptyset$ does not depend on the particular embedding.\end{Remark}

Next, we show that non-positive curvature implies that the graph is ''nice``. The proof follows from straightforward calculation.

\begin{Lemma} \label{l:nonpos_cases}
\emph{Let $G$ be a   planar graph that is connected and  locally finite.
\begin{itemize}
\item [(1.)] If $\ka_C(G)\leq0$, then $G$ is simple, admits no terminal vertices and each extended edge is regular. If $\ka_C(G)<0$, then there are no extended edges.
\item [(2.)] If $\ka_V(G)\leq0$, then $G$ admits no terminal vertices and each extended edge is regular. If $\ka_V(G)<0$, then there are no extended edges.
\item [(3.)] If $\ka_F(G)\leq0$, then $G$ is simple.
\end{itemize}}
\end{Lemma}

We come now to the proof of Theorem \ref{t:nonpos}.

\begin{proof}[Proof of Theorem \ref{t:nonpos}.]
If $G$ is finite, then by the Gauss-Bonnet formula, it must admit some positive curvature. \\
Lemma~\ref{l:nonpos_cases} implies that the assumptions of Lemma~\ref{l:non-pos_CopyPaste} are satisfied. Hence, by Lemma~\ref{l:non-pos_CopyPaste}, we have $D(F)=\emptyset$ and, by Lemma~\ref{l:non-pos_subtess_DF}, we obtain
that $G$ is locally tessellating. Moreover, by Lemma~\ref{l:nonpos_cases}, in the case of $\ka_C(G)<0$ or $\ka_V(G)<0$ the graph $G$ admits no extended edges and, thus, $G$ is  strictly locally tessellating.
\end{proof}

\section{Embedding of a locally tessellating graph into a tessellation}\label{s:LocTessVsTess}

In this section, we construct an embedding of a locally tessellating graph into a tessellating supergraph, which leaves crucial properties of a subset of vertices invariant. This embedding allows us  to carry over many results for tessellations to planar graphs in the forthcoming sections. We start the section with an extension of a proposition of Higuchi \cite{H},  to planar graphs. Then, after presenting the construction, we will extract some important properties of the tessellating supergraph.

\begin{Proposition}\label{l:Higuchi} \emph{Let $G$ be a simple,  planar graph that is connected and locally finite. If $\ka_{V}^{G}(v)<0$ for all $v\in V$, then $\ka_{V}(G)\leq -1/1806$. The maximum is achieved for vertices  with degree $3$ that have adjacent faces with degrees exactly $3$, $7$ and $43$.}
\begin{proof}
By Theorem~\ref{t:nonpos} every non-negatively curved graph is locally tessellating. In \cite[Proposition~2.1]{H} it is shown that $\ka_V^G<0$ on $V$ implies $\ka_V^G\leq-1/1806$ on $V$   for tessellating graphs and that the maximum is achieved as it is claimed above. The proof consists of a list of all relevant cases. The list is ordered by the vertex degree $n$ and a vector $(l_1,\ldots,l_n)$ where $l_1\leq\ldots\leq l_n$ are the degrees of the faces adjacent to the vertex.
Since we allow for unbounded faces, we have to check some additional cases: For $n\ge5$, the curvature for $(l_1,\ldots,l_4,\infty)$ with $l_1,\ldots,l_4\ge3$ is smaller or equal to $-1/6$.
For $n=4$, the curvature for $(3,3,3,\infty)$ is zero and for $(l_1,l_2,l_3,\infty)$ with $l_1,l_2\geq3$ and $l_3>3$ the curvature is smaller or equal to $-1/12$.  For $n=3$, the curvature of $(3,6,\infty)$ and $(4,4,\infty)$ is zero, while the curvature of $(l_1,l_2,\infty)$ with $l_1\geq 3$, $l_2>6$ is smaller or equal to $-1/42$ and with $l_1\geq 4$, $l_2>4$ it is smaller or equal to $-1/20$.
%Hence, the minimum is the same as in the case of tessellations, so, the assertion follows from \cite[Proposition~2.1]{H}.
\end{proof}
\end{Proposition}

We now come to the construction of the embedding. Let $G=(V,E,F)$ be a simple,  locally tessellating graph that satisfies $\ka_{V}(G)\le0$.
Let $W\subseteq V$ be a finite set of vertices \emph{simply connected}, i.e., both subgraphs $G_W$ and $G_{V\setminus W}$ are connected. The construction consists of two steps. In the first step, we add binary trees to certain vertices. In the second step, we close the unbounded faces by adding ''horizontal`` edges.

\textbf{Step 1:} To any vertex $v$ with $\ka_{V}^{G}(v)=0$ that is adjacent to $n\geq1$ infinigons, we attach $n$ binary trees. To do so, we  embed every one of these tree into a different infinigon and then connect the roots of the trees and $v$ by edges.
%For $v$ be an inner vertex of an extended edge and  $f$, $g$ the corresponding adjacent faces. Since in a locally tessellating graph every extended edge is regular, the faces $f$ and $g$ are unbounded. We attach two rooted  binary regular trees by connecting $v$ and  each of the roots by an edge. We embed the two trees into the plane such that one of them lies in $f$ and the other one in $g$. Similarly, we attach one rooted binary tree to a vertex $v$ that is adjacent to an infinigon $f$ and that satisfies $\ka_{V}^{G}(v)=0$, by embedding the tree into $f$. We apply this procedure to all such vertices.
With slight abuse of notation, we denote the face set of the resulting graph also by $F$.

\textbf{Step 2:} We choose the closing parameter, that is the size at which unbounded faces are closed by a ''horizontal`` edge. Let
$\mbox{diam}(W):=\max_{v,w\in W}d(v,w)$. For $\vareps>0$ define
$$R_\eps:= \max\ac{6,2{\mbox{diam}}(W),(2+\min_{v\in V}|v|)\frac{1}{\eps}}.$$
By induction over $n\in\N$, we perform the following procedure: For every unbounded face $f\in\dd_F B_n(W)$, we connect the two vertices in $f\cap S_n(W)$ by an edge whenever $\av{f\cap B_n(W)} > R_\eps$. (Note that the uniqueness of the two vertices follows since extended edges are regular.)  We denote, with slight abuse of notation, the face set of the modified graph after each induction step again by $F$.

This yields a limiting graph that we denote by $G'=G'_\eps=(V',E',F')$. Obviously, $G'$ is a super-graph of $G$, i.e., $V\subseteq V'$, $E\subseteq E'$. Therefore, it is natural to talk about corresponding vertices in $V$ and $V'$. In particular, we will denote for a vertex $v\in V$ the corresponding vertex in $V'$ by $v'$ and for a subset $W\subseteq V$ we denote the corresponding subset in $V'$ by $W'$.

\begin{Thm}\label{t:embedding} \emph{Let $G$ be a simple, connected, locally tessellating graph that satisfies $\ka_{V}(G)\leq0$, $W\subset V$ be finite and simply connected and $\eps>0$. Then, the graph $G'=G_\eps'$ constructed above is a tessellation and satisfies the following assertions:
\begin{itemize}
\item [(G1)] If $v\in W$ is not adjacent to an infinigon or $\ka_{V}^{G}(v)<0$, then $|v|=|v'|$. Otherwise, $|v'|=|v|+n$, where $n$ is the number of adjacent infinigons. Moreover, if
    $\ka^{G}_{C}(v,f)\le0$ for all $(v,f)\in C_v(G)$, then edges are added to $v$ if and only if $v$ is the inner vertex of an extended edge.
\item [(G2)] The embedding of $G$ into the supergraph $G'$ is a graph isomorphism of the subgraphs $G_W$ and $G_{W'}$, (i.e., the adjacency relations of corresponding vertices $W$ and $W'$ remain unchanged). If $\ka_{V}^{G}<0$ on  $W$, then the embedding is even a graph isomorphism of the subgraphs $G_{B_1^G(W)}$ and $G_{B_1^{G'}(W')}$.
\item [(G3)] The distance of two vertices $v,w\in W$  in $G$ equals the distance of the corresponding vertices $v',w'\in W'$ in $G'$.
\item [(G4)] If $\ka_C(G)\leq 0 $, then $\ka_C(G')\leq \min\{ 0,\ka_C(G)+\vareps\}$.
\item [(G5)] If $\ka_V(G)\leq 0 $, then $\ka_V(G')\leq    \min\{0, \ka_V(G)+\vareps\}$ whenever $\eps\in(0,1/1806)$.
\end{itemize}}
\end{Thm}
\begin{proof}It is obvious from the construction that $G'$ is a tessellating graph.

(G1):  Edges are added to vertices in $W$ only in Step~1. This is exactly the case,  if the vertex $v$ is adjacent to at least one infinigon and $\ka_{V}^{G}(v)=0$. Then, as many edges are added as there are adjacent infinigons. If $\ka^{G}_{C}(v,f)\le0$ for all $(v,f)\in C_{v}(G)$ and $\ka_{V}^{G}(v)=0$  then $\ka_{C}^{G}(v,f)=0$ for all $(v,f)\in C_{v}(G)$. If  $v$  is adjacent to an infinigon, then $|v|=2$ and both adjacent faces are infinigons.

(G2): The first statement follows since we do not connect or disconnect vertices within $W$. The second one follows from the first one and (G1).

(G3): Note that, in Step~2, we add edges and create paths in $G'$ that are not in $G$.  By definition of $R_{\eps}$, such a path in $G'$ connecting vertices in $W'$ has at least the length $\mbox{diam}(W)$. Therefore, the distance of  $v',w'\in W'$ in $G'$ is at least the distance of $v,w\in W$ in $G$. Moreover, by (G2) the distance does not increase either.

(G4): We consider three types of corners: Firstly, consider $(v',f')\in C(G')$ that is the corresponding corner of some $(v,f)\in C(G)$ with $|f|<\infty$. Clearly, $\ka_{C}^{G}(v,f)\leq\ka_{C}^{G'}(v',f')$ by (G1). Secondly, let $(v',g)$ be such that $v'$ is the corresponding vertex of some vertex $v\in V$ and $g$ is created in Step~2 by closing an infinigon. Obviously, $|v'|\geq\max\{3,|v|\}$ and $|g|\geq\max\{6,1/\eps\}$ by construction and the definition of $R_\eps$. Therefore,
$\ka_C^{G'}(v',g)\leq \min\ac{{1}/{3},{1}/{|v|}} -{1}/{2}+\min\ac{{1}/{6},\vareps} \leq \min\ac{0,\ka_C(G)+\vareps}.$
Thirdly, let $(w,g)\in C(G')$  be a corner of a vertex $w$ that was added with a binary tree in Step~1. In this case, $\ka_{C}(G)=0$. Obviously, $|w|\geq 3$ by construction and $|g|\geq 6$ by definition $R_\eps$. Therefore, $\ka^{G'}_C(w,g)\leq 1/3-1/2+ 1/6=0=\ka_C(G).$
%Hence, we have checked  (G4)  in all possible cases.

(G5): We consider three cases:
Firstly, let $v'\in V'$ be the corresponding vertex of some vertex $v$ to which $n\geq1$ binary trees were added in Step~1. Then, $n=|v'|-|v|$ by (G1). This gives a total of $2n$ new faces $g\in F'$ each of which has face degree  $|g|\geq6$ after being closed in Step~2. Therefore,
\begin{align*}
\ka_V^{G'}(v')\leq1-\frac{|v'|}{2}+\sum_{f\in F,v\in f,|f|<\infty}\frac{1}{|f|}+\frac{|v'|-|v|}{3}= \ka_V^G(v)-\frac{|v'|-|v|}{6}\leq\ka_V^G(v).
\end{align*}
Secondly, let $v'\in V'$ be the corresponding vertex of some $v\in V$ to which no edges were added in Step~1. If $v$ is not adjacent to an infinigon, then no edges are added in Step~2 either. Thus, $\ka_{V}^{G}(v)=\ka_{V}^{G'}(v')$. Otherwise, $\ka_{V}^{G}(v) <0$ (as, otherwise, edges were added in Step~1). Due to planarity, at most two edges were added to $v$ in Step~2, i.e., $|v|\leq|v'|\leq|v|+2$.
Moreover, for a face $g\in F'$ adjacent to $v'$, there is either a corresponding face $f\in F$ and $|g|=|f|$ or $g$ was created in Step~2 from an infinigon, in which case $|g|\geq\max\{6,(\min_{v\in V}|v|+2)/\eps\}=:1/\delta$ by definition of $R_\eps$.
By these considerations, we get
\begin{align*}
\ka_V^{G'}(v')&\leq 1-\frac{|v|}{2}+\sum_{f\in F,v\in f}\frac{1}{|f|}+|v'|\delta \leq \ka_{V}^{G}(v)+(|v|+2)\de\leq \ka_{V}^{G}(v)+\eps,
\end{align*}
where the last inequality follows by the definition of $\de$.
We have $\ka_{V}^{G'}(v') <0$ as $\ka_{V}^{G}(v) \leq-\eps$  by Proposition~\ref{l:Higuchi} whenever  $\eps\in(0,1/1806)$.
Thirdly, let $w\in V'$ be a vertex that has no corresponding vertex in $V$, i.e., it is a vertex of a binary tree which was added in Step~1. Note that, in this case, $\ka_V(G)=0$  and $|w|\geq 3$ by Step~1 and Step~2. Moreover, by definition of $R_{\eps}$, all faces $g\in F'$ adjacent to $w$ satisfy $|g|\geq 6$. We get $\ka^{G'}_V(w)\leq 1-|w|/2+|w|/6\leq0=\ka_V(G)$.
\end{proof}

\section{Geometric Applications}\label{s:GeoApp}

In this section, we discuss some applications of the fact that every non-positively curved planar graph is locally tessellating. Indeed, by the embedding constructed in the previous section, most of statements for locally tessellating graphs are now direct consequences of the results for tessellations.
These  results concern the absence of cut locus, the boundary structure of distance balls, estimates for the growth of distance balls, bounds and positivity of Cheeger's constant and empty interior of minimal bigons.

\subsection{Absence of cut locus}\label{s:no_cut_locus}

The \emph{cut locus} of a vertex $v_0$ of a graph is the set of all vertices, where the distance function $d(v_0,\cdot)$ attains a local maxima. In contrary, empty cut locus for all vertices implies that geodesics can be continued ad infinitum.

For non-positively curved tessellation, a corresponding result can be found in \cite{BP2}. Note that in \cite{BP1,BP2} the results and proofs are given for the metric space considering the distance function on the faces of the graph. However, the results are true for the metric space of vertices as well, since the dual graph of a non-positively corner curved tessellation is again a non-positively corner curved tessellation. %For a discussion of this fact see \cite{BP1}.

\begin{Thm} \label{t:no_cut_locus}
\emph{Let $G =(V,E,F)$ be a planar graph that is connected, locally finite and satisfies $\ka_C(G)\leq 0$. Then, the metric space $(V,d)$ has no cut locus.}
\begin{proof}
By Theorem~\ref{t:nonpos}, the graph $G$ is locally tessellating.
Suppose there is $v_0\in V$ with non-empty cut locus and suppose $v\in V$ is in the cut locus of $v_0$. Then, by the definition of the cut locus, all adjacent vertices of $v$ have smaller or equal distance to $v_0$. Consider a simply connected set of vertices $W$ that contains all vertices of the paths of minimal length from $v_0$ to $v$ and all adjacent vertices of $v$. By (G2), (G3) of Theorem~\ref{t:embedding}, there is a tessellation $G'$ such that the distances of corresponding vertices in $W$ and $W'$ agree.  Let $v_0',v'\in W'$ be the corresponding vertices  to $v_0,v$ in $W$. Then, $v'$ is in the cut locus of $v_0'$ in $G'$. By (G4), we have $\ka_C({G'})\leq 0$.  This leads to a contradiction to \cite[Theorem~1]{BP1} which guarantees absence of cut locus for tessellations under the assumption of non-positive corner curvature.
\end{proof}
\end{Thm}

\subsection{The boundary of distance balls}\label{s:locstructure}

Non-positive corner curvature has very strong implications on the boundary structure of distance balls. In particular, the concept of admissibility introduced in \cite{BP1,BP2} captures important aspects of the boundary behavior. Since this concept is quite involved, we only derive some of its most important consequences.

In \cite{Z}, some of these statements were already proven under various assumptions which all imply $\ka_{C}(G)<0$. In particular, these statements are used there to prove positivity of Cheeger's constant.
Here, we will use these properties to prove absence of finitely supported eigenfunctions of nearest neighbor operators on planar non-positively curved graphs.

\begin{Thm} \label{t:admissiblity} \emph{Let $G $ be a planar graph that is locally finite and satisfies $\ka_C(G)\leq 0$. Let  $v_0\in V$, $n\in\N_0$ and denote $B_n:=B_n(v_0)$ and $S_n:=S_n(v_0)$.
\begin{itemize}
  \item [(1.)] Every vertex in $S_{n}$ is adjacent to at least one vertex in $S_{n+1}$.
  \item [(2.)] Every vertex in $S_{n+1}$ is adjacent to at most two vertices in $S_{n}$.
 \item [(3.)] If two vertices in $S_{n}$ have a common neighbor in $S_{n+1}$, then
     both of them have another neighbor in $S_{n+1}$.
  \item [(4.)] Let $f_1,\ldots,f_{2k}$ be a cyclic enumeration of the faces of $\dd_F B_n$. Then, the case
       $\av{f_{2j-1}\cap  B_n}=\av{f_{2j}\cap V\setminus B_n}=1$ for all $1\leq j\leq k$ can not happen.
  \item[(5.)] The sphere $S_n$ admits a cyclic enumeration in the sense that two succeeding vertices are adjacent to a common boundary face in $\dd_F B_n$.
\end{itemize}}
\begin{proof}
Let $G'$ be the tessellating graph constructed from $B_{n+1}$ in $G$. As the embedding does not change distances of vertices in  $B_{n+1}$, by (G3). the spheres $S_n,S_{n+1}\subset V$ can be considered as subsets of the spheres $S_n':=S_n(v_0')$ and $S_{n+1}':=S_{n+1}(v_0')$ in $G'$.

(1.) Suppose $v\in S_n$ is not adjacent to any vertex in $S_{n+1}$. Then, $v$ is in the cut locus of $v_0$ which is a contradiction to Theorem~\ref{t:no_cut_locus}.

(2.) Suppose $v\in S_{n+1}$ is adjacent to more than two vertices in $S_n$. Then, the corresponding vertex $v'$ of $v$ in $G'$ is connected to more than two vertices in $S_{n+1}'$. This gives a contradiction to \cite[Proposition~2.5 (a)]{BP2}, (where a corresponding statement is found for the dual graph).

(3.) Let $u,v\in S_n$ be adjacent to some  $w\in S_{n+1}$. Denote by $u',v',w'$ the corresponding vertices in $V'$. The subgraph $G_{\{u,v,w\}}$ is a path and is not included in an extended edge (otherwise, this leads to a contradiction to $u,v\in S_n$ since every extended edge is regular). Hence, there is a unique face $f\in\dd_F B_n$ that contains $u,v,w$ and this face is bounded. By the construction of $G'$, the face $f$ has a corresponding boundary face $f'$ in $G'$. By \cite[Corollary~2.7]{BP2}, occurrence of such a face implies that the corresponding vertices $u', v'\in S_{n}'$ have neighbors $u_1', v_1'\in S_{n+1}'$ such that $u_1',v_1'\neq w'$. (In the language of \cite{BP2}, the dual vertex of the face $f$ has \emph{label $b$} and, therefore, its neighbors in the boundary have \emph{label $a^+$} by the \emph{admissibility of distance balls}. Translating this to our situation, we obtain the conclusion above.) Since $\ka_{C}(G)\leq0$ and the subgraph $G_{\{u,v,w\}}$ is not included in an extended edge, no edges were added to the vertices $u,v$ in the construction of $G'$, by (G1). Therefore, there are  $u_1,v_1\neq w$ in $S_{n+1}$ whose corresponding vertices are $u_1',v_1'$.

(4.) Assume the opposite. Then, $B_{n+1}$ encloses no infinigon. By (G1), no edges were added to $B_{n+1}$ while embedding it into a tessellation. The corresponding statement for tessellations, \cite[Proposition~13]{KLPS}, now gives a contradiction.

(5.) We have  $\ka_C(G')\leq 0$, by (G4). Thus,  $S_{n}'$ admits a cyclic enumeration, by \cite[Theorem~{3.2}]{BP1}. Since $S_n$ can be considered as a subset of $S_{n}'$, this gives enumeration of $S_n$. Now, it can be easily seen that this is a cyclic enumeration of $S_n$.
\end{proof}
\end{Thm}

\subsection{Growth of distance balls}\label{s:growthDistanceBalls}
In this subsection, we give estimates for the exponential growth of distance balls in terms of curvature. A lower bound for tessellations is found in  \cite[Theorem~5.1]{BP1} and an  upper bound in \cite[Theorem~4]{KP}.

\begin{Thm} \label{t:growthDistanceBalls}
\emph{Let $G $ be a simple,  planar graph that is connected, locally finite, has no cut locus and satisfies $\ka_V(G)<~0$. Then, for all $v\in V$ and $n\geq 1$
$$|S_{n}(v)|\geq -2\ka_V(G)\frac{q}{q-1}|B_{n-1}(v)|,$$
where  $q:=\sup_{f\in F}|f|$ and ${q}/\ab{q-1}=1$ in the case $q=\infty$. Moreover, for the exponential growth rate $\mu:=\limsup_{n \to \infty} \frac{1}{n}\log |S_n(v)|$ and  $p:=\sup_{v\in V}|v|$, one has
$$\log\ab{1-2\ka_V(G)\frac{q}{q-1}} \leq\mu\leq \log(p-1).$$}
\begin{Remark} By Lemma~\ref{l:nonpos_cases}.(1.) and Theorem~\ref{t:no_cut_locus}, the assumptions that $G$ is simple and has no cut locus is implied by $\ka_C(G)\leq0$.
\end{Remark}
\begin{proof}[Proof of Theorem~\ref{t:growthDistanceBalls}]
Theorem~\ref{t:nonpos} implies that $G$ is strictly locally tessellating. By the negative vertex curvature, (G2) and (G3) imply that the distance-$n$-ball of a vertex in $G$ is isomorphic to the distance-$n$-ball for the corresponding vertex in a tessellating graph $G'$. Therefore, $|S_n^{G}(v)|=|S_n^{G'}(v')|$ and $|B_n^{G}(v)|=|B_n^{G'}(v')|$.
By (G5),  we have $\ka_V(G')\leq \ka_V(G)+\vareps$ for  $\vareps\in(0,1/1806)$. Combining this with the statement for tessellations, \cite[Theorem~5.1]{BP1}, we obtain
\begin{eqnarray*}
 |S_n^{G}(v)|=|S_{n}^{G'}(v')|&\geq& -2\ka_V(G')\frac{q}{q-1}|B_{n-1}^{G'}(v')|\\
   &\geq& -2(\ka_V(G)+\vareps)\frac{q}{q-1}|B_{n-1}^{G}(v)|
\end{eqnarray*}
As $\vareps$ can be chosen arbitrarily small, we obtain the first result. The lower bound in the second statement is a direct consequence of the first one (compare  \cite[Corollary~5.2]{BP1}). The upper bound follows from a comparison to a $p$-regular tree. For more details see \cite[Theorem~4]{KP}.
\end{proof}
\end{Thm}

\subsection{Estimates for the Cheeger constant}\label{s:cheeger}

An isoperimetric constant known as the Cheeger constant plays an important role in many areas of geometry, probability and spectral theory. For infinite graphs, it was first defined by Dodziuk \cite{D} and later in another version by Dodziuk/Kendall \cite{DKe}. These different version appear in connection to different versions of the discrete Laplace operator. In \cite{D} and \cite{DKe}, the respective constant is used to estimate the bottom of the spectrum. In probability, positivity of Cheeger's constant implies that the simple random walk is transient. This and various other implications can be found in \cite{G,W,W2}.

For a subset $U\subseteq V $, let the Cheeger constants be defined as
\begin{eqnarray*}
\alpha_U&:=&\inf\ac{\frac{|\dd_E
W|}{\mbox{vol}(W)}\mid W\subseteq U\; \mathrm{finite}}\qand \al:=\al_V,\\
\beta_U&:=&\inf\ac{\frac{|\dd_E W|}{|W|}\mid W\subseteq U\; \mathrm{finite}}\qand\be:=\be_V,
\end{eqnarray*}
where $\dd_E W$ is the set of edges that connect a vertex in $W$ with a vertex in $V\setminus W$ and $\mbox{vol}(W)=\sum_{v\in W}|v|$. The constant $\al$ was first introduced in \cite{DKe} and $\be$ in \cite{D}. The set $\mathcal K$ of finite subsets of $V$ forms a net with respect to the inclusion relation. We define the following limits along this net
$$\al_{\dd V}:=\lim_{K\in \mathcal K}\al_{V\setminus K}\qand\be_{\dd V}:=\lim_{K\in \mathcal K}\be_{V\setminus K}.$$
In \cite{F} the quantity $\al_{\dd V}$ was introduced as $\al_\infty$.

Woess \cite{W} and $\dot{\mbox{Z}}$uk \cite{Z} proved separately that negative curvature implies a strong isoperimetric inequality, i.e., have positive Cheeger constant. While \cite{W} assumes that the graph is  tessellating and an average curvature is negative, \cite{Z} allows for infinigons and his assumptions imply negative corner or face curvature.
Positivity of Cheeger's constant was also proven later by Higuchi \cite{H} under the stronger assumption of negative vertex curvature in the case of tessellations. Explicit  formulas for the Cheeger constant of regular tessellations can be found  \cite{HJL,HiShi}. In \cite{KP} lower bounds for both types of Cheeger's constant are obtained in the context of locally tessellating graphs in terms of curvature.
Moreover, Fujiwara \cite{F} proved that the Cheeger constant at infinity  $\al_{\dd V}$ is equal to one for trees with vertex degree (and hence curvature) tending to negative infinity. In \cite{K}, it is shown that this implication holds also for  tessellating graphs.

\begin{Thm}%[Cheeger constant estimate]
\label{thm:cheegest}
\emph{Let $G$ be a simple, planar graph that is connected, locally finite and satisfies $\ka_{V}(G)\leq0$. \\
\emph{(1.)} For all  $U\subseteq V$, we have
\begin{align*}
\al_U \ge 1-\frac{1}{p_U}\frac{2q_U}{q_U-2} \qquad\mbox{ and }\qquad\be_U \ge  p_U-\frac{2q_U}{q_U-2},
\end{align*}
where  $p_U:=\inf\limits_{v\in U}|v|$, $q_U:=\inf\limits_{f\in F,f\cap U\neq\emptyset}|f|$ and the conventions $1/\infty=0$, $\infty/\infty=1$.\\
\emph{(2.)} If $\lim\limits_{K\in \mathcal K} \inf\limits_{v\in V\setminus K}\ka^G_V(v)=-\infty$, then $\al_{\dd V}=1$ and $\be_{\dd V}=\infty$.\\
\emph{(3.)}  We have
\begin{align*}
\al \ge -2 C  \sup\limits_{v \in V} \frac{1}{|v|}\ka_V^G(v) \quad\mbox{and}\quad \be \ge -2 C \ka_V({G}),
\end{align*}
where $C:=({1+\frac{2}{Q-2}})({1+\frac{2}{(P-2)(Q-2)-2}}),$ with $P:=\sup_{v\in V}|v|$, $Q:=\sup_{f\in F}|f|$ and  the conventions $1/\infty=1/(0\cdot\infty-2)=0$.\\
\emph{(4.)}  $\be\geq\al>0$, whenever $\ka_V^G<0$ on $V$.}
\end{Thm}

\begin{Remark}
(a.) Let $\gm:=(q-2)/2q$ be the normalized angle of a regular $q$-gon and $\ov\ka:=1-p\gm$. We can reformulate the estimates of (1.) for $U=V$ in terms of curvature as follows
\begin{align*}
    \al\geq \frac{-\ov\ka}{1-\ov\ka}\quad \mbox{and}\quad \be\geq-\frac{\ov\ka}{\gm}.
\end{align*}
While these estimates are new, (2.) is an extension of \cite{F,K}, (3.) is an extension of \cite{KP} and (4.) is a version of \cite{H,W,Z}. Note also that the first inequality of (4.) is independent of the assumption $\ka_V^G<0$.

(b.) For the proof of (1.) and (2.), we do not need the embedding of Section~\ref{s:LocTessVsTess}. This is important as $p_U$ and $q_U$ might be changed by Step~1 and Step~2.
\end{Remark}

\begin{proof}[Proof of Theorem~\ref{thm:cheegest}]
By Theorem~\ref{t:nonpos}, the graph $G$ is strictly locally tessellating.\\
(1.) We claim that
$$|\dd_E W|\geq \mbox{vol}(W)-\frac{2q_U}{q_U-2}(|W|+c(W)-2),$$
where $c(W)$ is the number of connected components of $G_{U\setminus W}$. For the proof of the formula, we follow the lines of the proof of \cite[Lemma~1]{K}, only instead of using the estimate $|f|\geq 3$, we go with $|f|\geq  q_U$  for all faces $f\in F$ with $|f\cap U |\neq\emptyset$. (Compare also to Proposition~2.2 and 2.3 in \cite{KP}.)
We obtain for all finite and connected sets $W\subseteq U$ with $c(W)\leq 2$
\begin{align*}
  \frac{|\dd_E W|}{\mbox{vol}(W)} &\geq 1-\frac{2q_U}{q_U-2}\frac{|W|}{\mbox{vol}(W)} \geq   1-\frac{1}{p_U}\frac{2q_U}{q_U-2},\\
  \frac{|\dd_E W|}{|W|} &\geq \frac{\mbox{vol}(W)}{|W|}-\frac{2q_U}{q_U-2}\geq {p_U}-\frac{2q_U}{q_U-2}.
\end{align*}

By the same arguments as in the proof of Proposition~6 in \cite{K}, it suffices to consider finite, connected sets $W\subseteq U$ with $c(W)\leq 2$. Thus, the formulas above yield (1.).

(2.) By the formula for the curvature of Lemma~\ref{l:non-pos_subtess_DF}, we have $\inf_{v\in V\setminus K}\ka^G_V(v)\to-\infty$ if and only if $p_{V\setminus K}\to\infty$ along the net $K\in\mathcal K$. Hence, (2.)  follows from (1.).

(3.) Theorem~1 of \cite{KP} states (3.) for locally tessellating $G$ under the assumptions that the right hand sides are positive. By Proposition~\ref{l:Higuchi}, $\ka_V(G)<0$ as $\ka_V^G<0$ on $V$. Moreover, one checks that $\ka^G_V(v)/|v|\leq 1/|v|-1/2+1/3\leq -1/42$ for $|v|> 6$ since $|f|\geq3$ for all $f\in F$. Therefore, also $\inf_{v\in V}\ka^G_V(v)/|v|<0$ if $\ka_V^G<0$ on $V$

(4.) The statement follows directly from (3.), as the right hand side is positive by the considerations in the proof of (3.).
\end{proof}

\subsection{Empty interior of minimal bigons}\label{s:bigons}

In this section, we discuss a geometric property that is related to hyperbolicity. In \cite{P}, it is shown for Cayley graphs of discrete groups that empty interior of minimal bigons is equivalent to Gromov hyperbolicity. Since we allow for arbitrary large faces, this equivalence is not true in our context. However, in \cite[Corollary~1]{Z}, Gromov hyperbolicity is shown by proving empty interior of minimal bigons under various assumptions implying negative corner curvature and the assumption of a uniform bound on the degree of polygons, see also \cite[Theorem~2]{BP2}.
Despite of that, if only all minimal bigons have empty interior, then one still can construct the Floyd-boundary  of $G$ and show that it is homeomorphic to $\Sp^1$. For a detailed discussion and references see \cite{Ka}.

Let us introduce the notion of a minimal bigon. Let $p_1=(v_1,\ldots,v_n)$ and $p_2=(w_1,\ldots,w_n)$ be the vertices of two finite paths satisfying $d(v_1,v_n)=d(w_1,w_n)=n$ and
$v_1=w_1$, $v_n=w_n$. Such a pair $(p_1,p_2)$ is called a \emph{bigon}. A bigon is called \emph{minimal} if $v_j\neq w_j$ for $j\neq1,n$. The \emph{interior of a minimal bigon} are all vertices enclosed by the two paths that do not belong to any of them.

\begin{Thm} \label{t:bigons}
\emph{Let $G $ be a  planar graph that is connected, locally finite and satisfies $\ka_C(G)<~0$. Then, any minimal bigon has empty interior. Moreover, if there is uniform upper bound on the face degree of the polygons, then the graph is Gromov hyperbolic.}
\begin{Remark}Note that the assumption for the Gromov hyperbolicity  only excludes the existence of arbitrary large polygons but not the existence of infinigons.
\end{Remark}
\begin{proof} Since $\ka_C(G)<0$, the graph is strictly locally tessellating by Theorem \ref{t:nonpos}. Let $W$ be the union of the vertices of a bigon and its interior. Obviously, $W$ is simply connected.  Moreover, note that $\ka_C(G)<0$ implies $\ka_V(G)<0$. By (G2), (G3), (G4) there is a tessellation $G'$ with $\ka_C(G')<0$ such that the distance of vertices in $W$ remain unchanged compared to the corresponding set $W'$ in $G'$. Thus, $W'$ is a minimal bigon as well. By \cite[Theorem~2]{BP2}, any minimal bigon in $G'$ has empty interior and the statement follows from (G2). The statement about the Gromov hyperbolicity now follows from the arguments of \cite{P}.
\end{proof}
\end{Thm}

\section{Applications in Spectral Theory}\label{s:SpApp}

We start this section by introducing two well known versions of the discrete Laplace operator. Then, we recall the corresponding bounds for the bottom of the spectrum implied by the Cheeger constant. After that, we show that the essential spectrum for both versions of the Laplacian is trivial, whenever the curvature decreases uniformly to $-\infty$. This is an extension of \cite{F} and \cite{K}. Finally, extending \cite{KLPS}, we prove that nearest neighbor operators on  non-positively corner curved graphs have no finitely supported eigenfunctions.

Let $c(V)$ be the space of complex valued functions on $V$ and $c_c(V)$ the space of functions that are zero outside a finite set. The discrete Laplace operator $\Lp$, often used in mathematical physics, is acting as
$$(\Lp\ph)(v)=\sum_{w\sim v}(\ph(v)-\ph(w))$$
and is essentially self-adjoint  on $c_c(V)$, (for a proof see \cite{Woj}). We denote the self-adjoint extension on $\ell^2(V)=\{\ph\in c(V)\mid\sum_{v\in V}|\ph(v)|^2<\infty\}$  also by $\Lp$.
Another version of the Laplacian $\LF$, often used in discrete spectral geometry, acting as
$$(\LF\ph)(v)=\frac{1}{|v|}\sum_{w\sim v}(\ph(v)-\ph(w))$$
on $\ell^2(V,\av{\cdot})=\{\ph\in c(V)\mid\sum_{v}|v||\ph(v)|^2<\infty\}$ is a bounded, self-adjoint operator.

\subsection{Spectral bounds and triviality of essential spectrum}\label{s:specbounds}
The following bound for the bottom of the spectrum of $\LF$ can be derived from \cite{M}, see also \cite{BMS-T,F}
$$1-\sqrt{1-\al^2}\leq\inf\si(\LF).$$
Here, $\al$ is the Cheeger constant defined in Section~\ref{s:cheeger}. This extends  to a bound for the bottom of the spectrum of  the Laplacian $\Lp$ (see \cite{K}, compare also \cite{Woj})
$$(1-\sqrt{1-\al^2})\inf_{v\in V}|v|\leq\inf\si(\Lp).$$
Hence, the bounds  on $\al$  in Theorem~\ref{thm:cheegest} give bounds for the bottom of  the spectrum. The next theorem generalizes a result in \cite{K}, see also \cite{F}.

\begin{Thm}\label{t:rap}
\emph{Let $G $ be a simple, planar graph that is connected, locally finite and satisfies $\ka_V(G)\leq0$. Then
\begin{itemize}
  \item [(1.)] $\se(\LF)=\{1\}$ follows if $\ka^{G}_V(v_n)\to-\infty$ for $v_n\to\infty$,
  \item [(2.)] $\se(\Lp)=\emptyset$  if and only if $\ka_V^{G}(v_n)\to-\infty$ for $v_n\to\infty$,
\end{itemize}
where the limit $v_n\to\infty$ means that the sequence eventually leaves  every finite set.}
\begin{proof} Under the assumptions above,
Theorem~\ref{thm:cheegest}.(2.) yields $\al_{\dd V}= 1$. This implies $\se(\LF)=\{1\}$ by \cite[Theorem~1]{F}.
The equivalence in (2.) follows from \cite[Theorem~2]{K} and the fact that the vertex degree tends to $\infty$ if and only if the curvature tends to $-\infty$.
\end{proof}
\end{Thm}

\subsection{Absence of finitely supported eigenfunctions}\label{s:no_cpt_supp_ef}
A linear operator $A$ defined on a subspace of $c(V)$ is called a \emph{nearest neighbor operator} on $G$ if its matrix representation in the standard basis is given by some $a:V\times V\to\C$ such that $a(w,v)\neq0$ if $v\sim w$ and $a(w,v)=0$ if  $v\not\sim w$ and $v\neq w$. Hence, $A$ acts  as
$$(A\ph)(v)=\sum_{w\in V} a(v,w)\ph(w)=a(v,v)\ph(v)+\sum_{v\sim w} a(v,w)\ph(w).$$
The operators $\Lp$ and $\LF$ (possibly plus multiplication by a potential) are nearest neighbor operators. The following theorem is proven in \cite{KLPS} for tessellating graphs.

\begin{Thm} \label{t:no_cpt_supp_ef}
\emph{Let $G $ be a planar graph that is connected, locally finite and satisfies $\ka_C(G)\leq 0$. Then, a nearest neighbor operator on $G$ does not admit finitely supported eigenfunctions.}
\end{Thm}

The proof in \cite{KLPS} is based on an induction over the distance balls of the metric space of faces. Since we allow for unbounded faces, these distance balls might have infinite cardinality. Therefore, the proof can not be carried over directly. We will give an alternative proof which uses a representation of the operator in polar coordinates and makes use of what we know about the boundary of distance balls, Theorem~\ref{t:admissiblity}.

Let $G=(V,E,F)$ be a planar graph that is locally finite and satisfies $\ka_C(G)\leq0$. For $v_0\in V$, denote $S_n=S_n(v_0)$ and  $s_n=|S_{n}|$. By Theorem~\ref{t:admissiblity}.(5.), we have a cyclic enumeration of $S_n$.
We reorder the enumeration in the spheres inductively by cyclic permutation. Let $v_1^{(n)}$ be the first vertex in the enumeration of $S_n$. We shift the enumeration of $S_{n+1}$  such that in the new enumeration $v_1^{(n+1)}$ is the first vertex (with respect to the unshifted enumeration of $S_{n+1}$) that is adjacent to $v_{1}^{(n)}$. Inductively, we get an enumeration, $v_1^{(n)},\ldots, v_{s_n}^{(n)}$ for all spheres $S_n$.

For a function $\ph\in c(V)$, let
$\ph_n$ be the restriction to $c(S_n)$.
For a nearest neighbor operator $A$,  let the matrices $E_n\in \C^{s_{n+1}\times s_{n}}$,
$D_n\in\C^{s_{n}\times s_{n}}$, $E_n'\in\C^{s_{n}\times s_{n+1}}$ be given such that
$$(A\ph)_n=-E_{n-1}\ph_{n-1}+D_n\ph_n-E_n'\ph_{n+1},$$
Then, $D_n$ is  the restriction of $A$ to $c(S_n)$ and the matrices $E_n$ and $E_n'$ are given by $E_n(i,j)=a(v_i^{(n+1)},v_j^{(n)})$ and $E_n'(j,i)=a(v_j^{(n)},v_i^{(n+1)})$ for $i=1,\ldots,s_{n+1}$, $j=1,\ldots,s_{n}$. A similar construction was given in \cite{FHS}.

\begin{Lemma} \label{l:E_n}
\emph{Let $G $ be a  planar graph that is connected, locally finite and satisfies $\ka_C(G)\leq 0$. Then, for $n\in \N_0$, we have the following:
\begin{itemize}
\item [(1.)] Each column of $E_n$ has at least one non-zero entry.
\item [(2.)] Each row of $E_n$ has exactly one or two non zero entries. Two non-zero entries always correspond to succeeding vertices in the enumeration of $S_n$.
\item [(3.)] Each two columns of $E_n$ have at most one non-zero entry at the same component. In this case, each of the columns have another non-zero entry at a distinct component.
\end{itemize}}
\begin{proof}
Statement (1.) follows since each vertex in $S_{n}$ is connected to a vertex in $S_{n+1}$ by  Theorem~\ref{t:admissiblity}.(1.). Statement (2.) follows since each vertex in $S_{n+1}$ is connected to at most two vertices in $S_{n}$, by Theorem~\ref{t:admissiblity}.(2.). The other statement of (2.) follows from the enumeration of the distances spheres. Statement (3.) follows from the planarity and Theorem~\ref{t:admissiblity}.(3.) \end{proof}
\end{Lemma}

\begin{Lemma} \label{l:E_n_injective}
\emph{Let $G $ be a   planar graph that is connected, locally finite and satisfies
$\ka_C(G)\leq 0$. Then, $E_n$ is injective for all
$n\in \N_0$.}
\begin{proof} Suppose $E_n$ is not injective, i.e., its column vectors are linearly dependent.
By the preceding lemma, $E_n$  must be of the form
$${E}_n=\left(
            \begin{array}{lllll}
               a_n(1,1)& &    & a_n(1,s_{n}) \\
               a_n(2,1)& \ddots  &  &  \\
               &  \ddots&a_n(s_{n+1}-1,s_{n}-1)   &  \\
               &   & a_n(s_{n+1},s_n-1)& a_n(s_{n+1},s_{n})\\
            \end{array}
          \right),$$
where $a_n(i,j)=a(v_i^{(n+1)},v_j^{(n)})$, $i=1,\ldots,s_{n+1}$, $j=1,\ldots, s_n$ and all other entries are zero. However, this situation is geometrically impossible by Theorem~\ref{t:admissiblity}.(4.) \end{proof}
\end{Lemma}

\begin{proof}[Proof  of Theorem \ref{t:no_cpt_supp_ef}.]
Suppose $\ph\in c_c(V)$ is an eigenfunction of $A$ to $\lm\in\C$. Let $v_0\in V$ be such that $\ph(v_0)\neq 0$ and $k\in\N$ such that $\ph_{k-1}\not\equiv 0$ and $\ph_n\equiv 0$ for $n\geq k$.  Rewriting the eigenvalue equation $(A\ph)_k=\lm\ph_k$ on the $k$-th sphere, one has
$$E_{k-1}\ph_{k-1}=(D_k-\lm)\ph_k-E_k'\ph_{k+1}.$$
By the choice of $k$, the right hand side is equal to zero. Since $E_{k-1}$ is injective by Lemma~\ref{l:E_n_injective} and $\ph_{k-1}\not\equiv 0$, the left hand is non-zero. This is a contradiction.
\end{proof}

{\bf Acknowledgement:} The author would like to thank Daniel Lenz for suggesting some of the questions that initially inspired this research. The author is also grateful to Norbert Peyerimhoff for many helpful remarks and suggestions on an earlier version of this paper.
He acknowledges the financial support by the German Science Foundation (DFG) and the Klaus Murmann Fellowship Programme (SDW).

\end{document}